\documentclass[12pt]{amsart}
\usepackage{graphicx}
\usepackage{amssymb}
\usepackage{color}
\usepackage{setspace}
\newtheorem{theorem}{Theorem}[section]
\newtheorem{lemma}[theorem]{Lemma}
\newtheorem{corollary}[theorem]{Corollary}
\theoremstyle{definition}
\newtheorem{definition}[theorem]{Definition}

\theoremstyle{remark}
\newtheorem{remark}[theorem]{Remark}
\newtheorem{question}[theorem]{Question}

\begin{document}
\title[Cantor Sets with Complement having the same {\large$\pi_{1}$}]
{Inequivalent Cantor Sets in $R^{3}$ Whose Complements Have the Same Fundamental Group}
%    author one information
\author{Dennis J. Garity}
\address{Mathematics Department, Oregon State University,
Corvallis, OR 97331, U.S.A.}
\email{garity@math.oregonstate.edu}

%    author two information
\author{Du\v{s}an Repov\v{s}}
\address{Faculty of Mathematics and Physics, and Faculty of Education, University of Ljubljana, P.O.Box 2964,
Ljubljana, Slovenia 1001}
\email{dusan.repovs@guest.arnes.si}
\date{November 1, 2011}

\begin{abstract}
For each  Cantor set $C$ in $R^{3}$,  all points of which have bounded local genus, we show that there are infinitely many inequivalent Cantor sets in $R^{3}$  with complement having the same fundamental group as the complement of $C$. This answers a question from 
\emph{Open Problems in Topology} and has as an application a simple construction of nonhomeomorphic open $3$-manifolds with the same fundamental group. The main techniques used are analysis of local genus of points of Cantor sets, a construction for producing rigid Cantor sets with simply connected complement, and manifold decomposition theory. The results presented  give an argument that for  certain groups $G$, there are uncountably many  nonhomeomorphic open $3$-manifolds with fundamental group $G$.
\end{abstract}

\maketitle

\section{Introduction}

The following question was asked in \emph{Open Problems in Topology II}
(see Question 14 in \cite{GaRe07}):
% %
%%
\begin{question}\label{Q1Open}
Can two different (rigid) Cantor sets have complements with  the same
fundamental group?
\end{question}
% %
%%

There are two parts to the question. First, are there any Cantor sets satisfying the condition? Second, are there rigid such Cantor sets. Answering for rigid Cantor sets seems more difficult. We focus our attention on answering these questions in $R^{3}$. 
The same techniques apply to embeddings in $S^{3}$. 
Here, \emph{different} Cantor sets means Cantor sets that are inequivalently embedded in the ambient space. The following theorem from \cite{GaReZe06}
together with results on local genus, $g_{x}(X)$,  of points $x$ in a Cantor set $X$
(see Section \ref{GenusSection}) give a positive answer to the above question when 
the fundamental group of the complement is trivial.

\begin{theorem}\cite{GaReZe06}
\label{rigidtheorem} 
For each increasing sequence $S=(n_1, n_2, \ldots)$
of integers such that $n_1>2$, there exists a wild Cantor set,
$X=C(S)$, in
$ R^3$, and a countable dense subset
$A=\{a_1, a_2, \ldots\} \subset X$ such that the following
conditions hold:

\begin{enumerate}
\item $g_x(X) \le 2$ for every $x\in X\setminus A$,
\item $g_{a_i}(X)=n_i$ for every $a_i\in A$, and
\item $R^3\setminus X$ is simply connected.
\end{enumerate}
\end{theorem}

\begin{remark}\label{RididVariation} 
For later reference, we note that the construction in \cite{GaReZe06}
actually works for any sequence $S=(n_1, n_2, \ldots)$ of integers
where each $n_{i}\geq 3$. The condition that the sequence is increasing
was used in this earlier paper to prove rigidity of the Cantor sets.
\end{remark}

Since local genus is preserved by equivalence, any Cantor sets 
$X_{1}=C(S_{1})$ and $X_{2}=C(S_{2})$ corresponding to distinct increasing sequences 
$S_{1}$ and $S_{2}$ as above are inequivalent. Since there are uncountably many such sequences, there are uncountably many inequivalent Cantor sets with simply connected complement. Since all of these Cantor sets are rigid, this answers both of the 
above questions in the case when the complement is simply connected. 

The more interesting case is when a Cantor set $C$ in $R^{3}$ has nonsimply connected complement. 
Question \ref{Q1Open} asks whether there is an inequivalent Cantor set $D$ with complements having the same fundamental group. We answer this question affirmatively for Cantor sets of bounded local genus and for a large class of other Cantor sets. 
The essential ingredient is the class of Cantor sets provided by 
Theorem \ref{rigidtheorem} with genus of points taking values in carefully chosen sequences of positive integers and with simply connected complement.

\begin{theorem}[Main Theorem]
\label{maintheorem} 
Let $C$ be a Cantor set in $R^{3}$. Suppose there is some integer  $N\geq 3$ such that there are only finitely many points in $C$ of local genus $N$. Then there are uncountably many inequivalent Cantor sets $C_{\alpha}$ in $R^{3}$ with complement having the same fundamental group as 
the complement of $C$.
\end{theorem}

\begin{corollary}[Bounded genus]
\label{maincorollary} 
Let $C$ be a Cantor set in $R^{3}$ of bounded local genus, or more generally, a Cantor set where the local genus of points never takes on a specific integer value $N\geq 3$. Then there are uncountably many inequivalent Cantor sets $C_{\alpha}$ in $R^{3}$ with complement having the same fundamental group as 
the complement of $C$.
\end{corollary}
\begin{proof}
If $C$ is a Cantor set as described in the corollary, then there is an integer $N\geq 3$ as in the statement of Theorem \ref{maintheorem}.
\end{proof}

By considering the open $3$-manifolds that are the complements of the Cantor sets in the above results, we are able to show in Section 7 that these open $3$-manifolds have uncountably many associated nonhomeomorphic open $3$-manifolds with the same fundamental group. In particular, we show:

\begin{theorem} \label{endtheorem}
Let $M$ be an open $3$-manifold with end point (Freudenthal) compactification $M^{\ast}$ such that $M^{\ast}\cong S^{3}$ and such that $M^{\ast}\setminus M$ is a Cantor set $C$. If the Cantor set $C$ has an associated integer $N\geq 3$ such that only finitely many points of $C$ have local genus $N$, then there are uncountably many open $3$-manifolds not homeomorphic to $M$ that have the same fundamental group as $M$.
\end{theorem}

\section{Terminology}\label{DefSection}

A subset $A\subset R^n$ is said to be \emph{rigid} if whenever
$f\colon R^n\to R^n$ is a homeomorphism with $f(A)=A$ it
follows that $f|_A=id_A$.  There are many examples in $R^3$ of
wild Cantor sets that are either rigid or have simply connected
complement. In \cite{GaReZe06},  examples were constructed having
both properties. 

See the bibliography for more results on embeddings of Cantor sets.
In particular, see Kirkor \cite{Kirkor}, DeGryse and Osborne \cite{DeGryse},
Ancel and Starbird \cite{Ancel}, and Wright \cite{Wright4} for
further discussion of wild Cantor sets with simply connected
complement.

Two Cantor sets $X$ and $Y$ in $R^3 $ are said to be
\emph{topologically distinct} or \emph{inequivalent} if there is
no homeomorphism of $R^3 $ to itself taking $X$ to $Y$. Sher
proved in \cite{Sher1} that there  exist uncountably many
inequivalent Cantor sets in $R^3$. He showed that varying the
number of components in the Antoine construction leads to these
inequivalent Cantor sets.

A simple way of producing new Cantor sets in $R^{3}$ is to take the union of two disjoint Cantor sets in $R^{3}$. This leads to the following definition. A Cantor set in $R^{3}$ is said to be \emph{splittable} into Cantor sets $C_{1}$ and $C_{2}$ if
\begin{itemize}
\item $C=C_{1}\cup C_{2}$, $C_{1}\cap C_{2}=\emptyset$ and
\item There are disjoint tame closed $3$-cells $D_{1}$ and $D_{2}$ in $R^{3}$
with
$C_{1}\subset D_{1}$ and $C_{2}\subset D_{2}$.
\end{itemize}

Given a Cantor set $X$ in $R^3 $, we denote the fundamental group
of its complement by $\pi_{1}(X^{c})$. As in the following remark, Theorem \ref{rigidtheorem} could be used to produce splittable examples of Cantor sets $X\cup C$ as the union of disjoint sets    with 
$\pi_{1}(X^{c})\simeq \pi_{1}((X\cup C)^{c})$. 
The examples we produce in the present paper require a more careful construction and are not splittable in this fashion.

\begin{remark}
Note that if $C$ is splittable into $C_{1}$ and $C_{2}$, then a Seifert - Van Kampen argument shows that $\pi_{1}(C^{c})\simeq \pi_{1}(C_{1}^{c})\ast\pi_{1}(C_{2}^{c})$ where $\ast$ represents the free product. Thus if  $\pi_{1}(C_{1}^{c})$ is trivial, then $\pi_{1}(C^{c})\simeq \pi_{1}(C_{2}^{c})$.
\end{remark}

\section{Defining Sequences and Local Genus}\label{GenusSection}

The following definitions about genus are from
\cite{Ze05}.

A \emph{defining sequence} for a Cantor set $X\subset R^3$ is a
sequence $(M_i)$ of compact 3-manifolds with boundary such that

\begin{itemize}
\item[(a)] each $M_{i}$ consists of pairwise disjoint cubes with
handles;
\item[(b)] $M_{i+1}\subset \mathop{\rm Int}\nolimits M_{i}$ for
each $i$; and
\item[(c)] $X=\bigcap_{i}M_{i}$.
\end{itemize}

Let
${ \mathcal{D}}(X)$ be the set of all defining sequences for $X$.
It is known (see \cite{Armentrout1}) that every Cantor set in $R^{3}$ has a
defining sequence, but the sequence is not uniquely determined.
In fact, every Cantor set has many inequivalent (see
\cite{Sher1} for the definition) defining sequences.

Let $M$ be a handlebody. We denote the genus of $M$ by
$g(M)$. For a disjoint union of handlebodies
$M=\bigsqcup_{\lambda\in \Lambda} M_\lambda$, we define
$g(M)=\sup\{g(M_\lambda);\ \lambda\in\Lambda\} $.

Let $(M_i)\in{\mathcal{D}}(X)$ be a defining sequence for a
Cantor set $X\subset R^3$. For any subset $A\subset X$ we denote
by $M_i^A$ the union of those components of $M_i$ which intersect
$A$. Define
% %
%%
\begin{eqnarray*}
g_A(X;(M_i)) &=& \sup\{g(M_i^A);\ i\geq 0\}\ \ \mbox{ and} \\
g_A(X) &=& \inf\{ g_A(X;(M_i));\ (M_i) \in {\mathcal{D}}(X)\}.
\end{eqnarray*}
% %
%%
The number $g_A(X)$ is either a nonnegative integer, or $\infty$, and is called 
\emph{the genus of the Cantor set
$X$ with respect to the subset $A$}. For $A=\{x\}$ we call the
number $g_{\{x\}}(X)$  \emph{the local genus of the Cantor set
$X$ at the point $x$} and denote it by $g_x(X)$. For $A=X$ we
call the number $g_X(X)$ \emph{the genus of the Cantor set $X$}
and denote it by $g(X)$.

Let $x$ be an arbitrary point of a Cantor set $X$ and
$h\colon R^3\to R^3$ a homeomorphism. Then any defining
sequence for $X$ is mapped by $h$ onto a defining sequence for
$h(X)$. Hence the local genus $g_x(X)$ is the same as the local
genus $g_{h(x)}(h(X))$. Therefore local genus is an embedding invariant.

Determining the (local) genus of a given Cantor set using the
definition is not easy. If a Cantor set is given by a defining
sequence one can  determine an upper bound. 

A direct consequence of the definitions
is the following result. 

\begin{lemma}
\label{localgenuslemma}
The local genus of $x$ in $C$ is  a nonnegative integer $k$ if and only if: 
\begin{enumerate}
\item For each defining sequence $(M_{i})$  for $C$, 
 there exists a natural number $N$ such that if $n\geq N$, then 
 $g(M_{n}^{\{x\}})\geq k$, and
 \item There exists a defining sequence $(N_{i})$ for $C$ and  a natural number $M$
 so that if $i\geq M$, then $g(N_{i}^{\{x\}})=k$.
 \end{enumerate}
 The local genus of $x$ in $C$ is   $\infty$ if and only if:
\begin{itemize}
\item[] For each defining sequence $(M_{i})$  for $C$, and for every pair of
natural numbers $(j,k)$, there exists an interger $\ell\geq j$ with
 $g(M_{\ell}^{\{x\}})\geq k$. 
 \end{itemize}
\end{lemma}
% %
%%

\section{Wedges of Cantor Sets}\label{WedgeSection}

Every Cantor set in $R^{3}$ is contained in a closed round 3-cell. By shrinking the radius of this cell, one can find a 3-cell that contains the Cantor set and which has a point (or points) of the Cantor set in its boundary.

\begin{definition} 
Let $C$ be a Cantor set in $R^{3}$, and $B$ a tame $3$-cell in $R^{3}$ 
with $C\subset B$. If $C\cap \text{Bd} (B) \neq \emptyset$, $C$ is said to be
\emph{supported} by $B$.  
\end{definition}

The following result is a consequence of the definition of tameness.

\begin{lemma}
\label{SupportLemma}
{\rm(See Figure \ref{Supportfigure}.)} 
Suppose $C$ is a Cantor set in $R^{3}$ supported by $B$,  
and $x\in C\cap \text{Bd} (B)$. Then if $(D,p)$ is a pair consisting of a tame ball of 
some radius in $R^{3}$ together with a point in the boundary of this ball, then there is a homeomorphism from $R^{3}$ to itself taking $(B,x)$ to  $(D,p)$.
\end{lemma}

\begin{figure}[h]
\includegraphics[width=.25\textwidth]{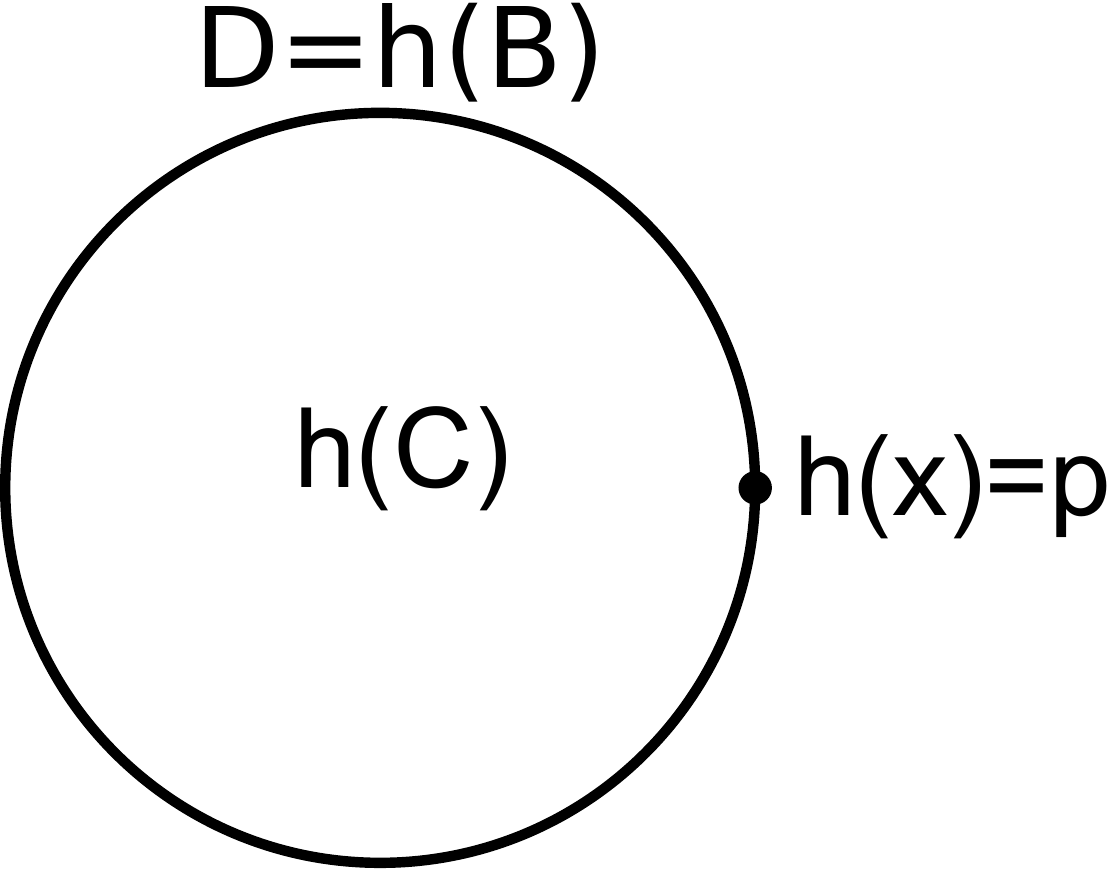}
\caption{Image of $B$ and $C$}
\label{Supportfigure}
\end{figure}

\newpage
\begin{definition}
\label{wedgedef}
(See Figure \ref{Wedgefigure}.) Let $C_{i}$  be a Cantor set in $R^{3}$ supported by $B_{i}$ with
$x_{i}\in \text{Bd} (B_{i})\cap C_{i}, i\in\{1,2\}$. The wedge of 
$C_{1}$ and $C_{2}$ at $x_{1}$ and $x_{2}$, $(C_{1},x_{1}) \bigvee (C_{2},x_{2})$, 
is defined as follows. Choose  orientation preserving self-homeomorphisms $h_{i}$  
of $R^{3}$ taking $(B_{i},x_{i})$ to $(D_{i},\textbf{0})$ where $D_{i}$ is the ball of radius 
$1$ about the point $\left(2*(i-\frac{3}{2}),0,0\right)$ in $R^{3}$. Then
\[
(C_{1},x_{1}) \bigvee (C_{2},x_{2})\equiv h_{1}(C_{1})\cup h_{2}(C_{2}).
\]

\end{definition}

\begin{figure}[h]
\includegraphics[width=.35\textwidth]{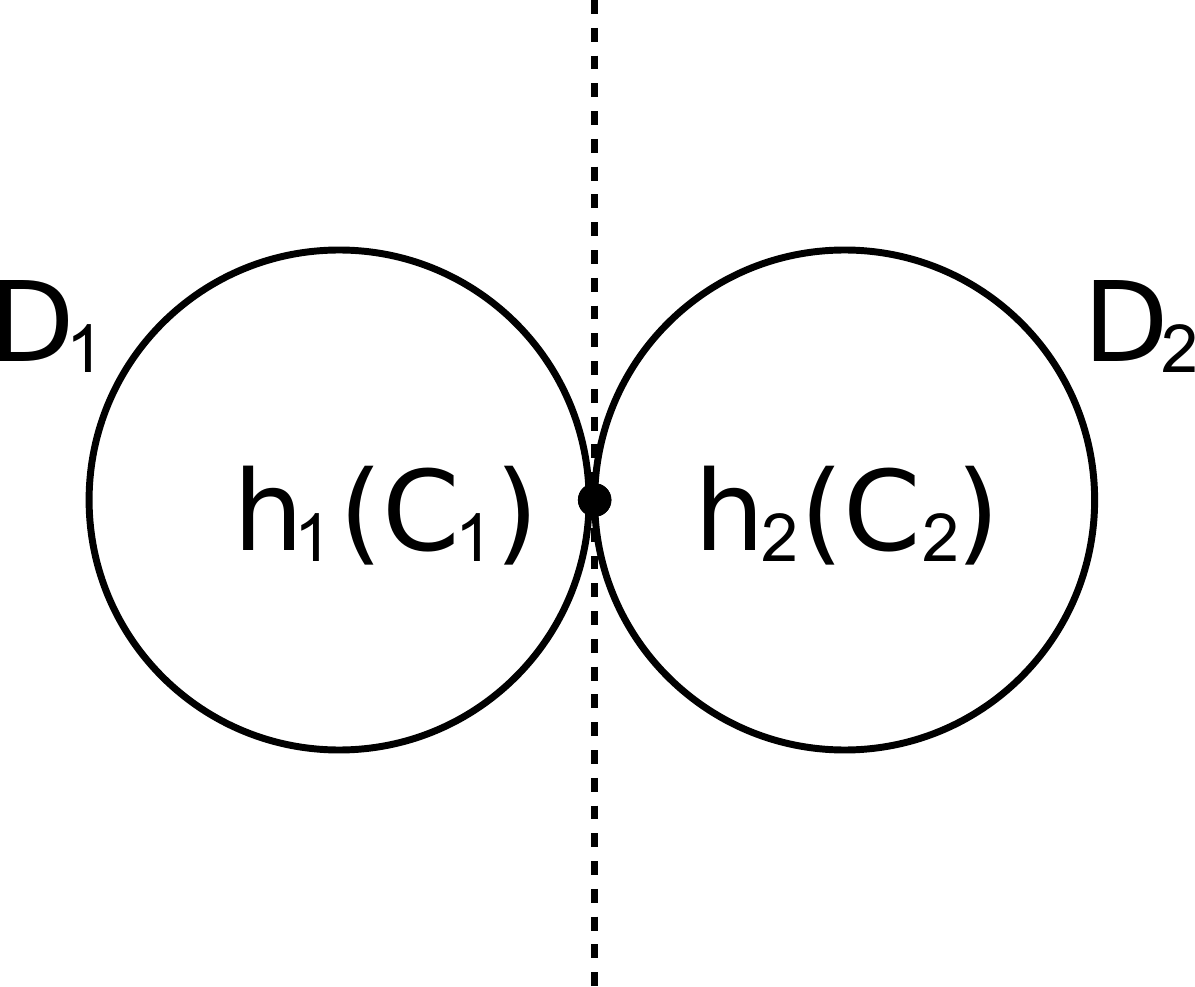}
\caption{Wedge of Cantor Sets}
\label{Wedgefigure}
\end{figure}

\begin{remark} Note that $(C_{1},x_{1}) \bigvee (C_{2},x_{2})$ is a Cantor set with
sub-Cantor sets $h_{i}(C_{i})$ embedded in $R^{3}$ in an equivalent manner to $C_{i}$.
\end{remark}

To make the proof of the next theorem and a result in the next section easier, we provide an alternate 
definition of the wedge of two Cantor sets that is equivalent to the definition above.

\begin{definition}
\label{wedgedef2}
[Alternate Definition of Wedge]
(See Figure \ref{altwedge}.) Let $C_{i}$  be a Cantor set in $R^{3}$ supported by $B_{i}$ with
$x_{i}\in \text{Bd} (B_{i})\cap C_{i}, i\in\{1,2\}$. The wedge of 
$C_{1}$ and $C_{2}$ at $x_{1}$ and $x_{2}$, $(C_{1},x_{1}) \bigvee^{\prime} (C_{2},x_{2})$ 
is defined as follows. Choose  orientation preserving self homeomorphisms $k_{i}$  
of $R^{3}$ taking $(B_{i},x_{i})$ to $(D^{\prime}_{i},(2*(i-\frac{3}{2}),0,0)$, where $D^{\prime}_{i}$ is the ball of radius 
$1$ about the point $(4*(i-\frac{3}{2}),0,0)$ in $R^{3}$. Let $A$ be the straight arc
from $(-1,0,0)$ to $(1,0,0)$ in $R^{3}$. Let $p:R^{3}\rightarrow R^{3}/A$ be the quotient map.
Then 
\\[6pt]
\centerline{$
 (C_{1},x_{1}) \bigvee^{\prime} (C_{2},x_{2})\equiv p\left(\strut k_{1}(C_{1})\cup k_{2}(C_{2})\right)
\equiv p\left(\strut k_{1}(C_{1})\cup A\cup k_{2}(C_{2})\right).
$}
\end{definition}
% %
%%
\begin{figure}[h]
\includegraphics[width=.45\textwidth]{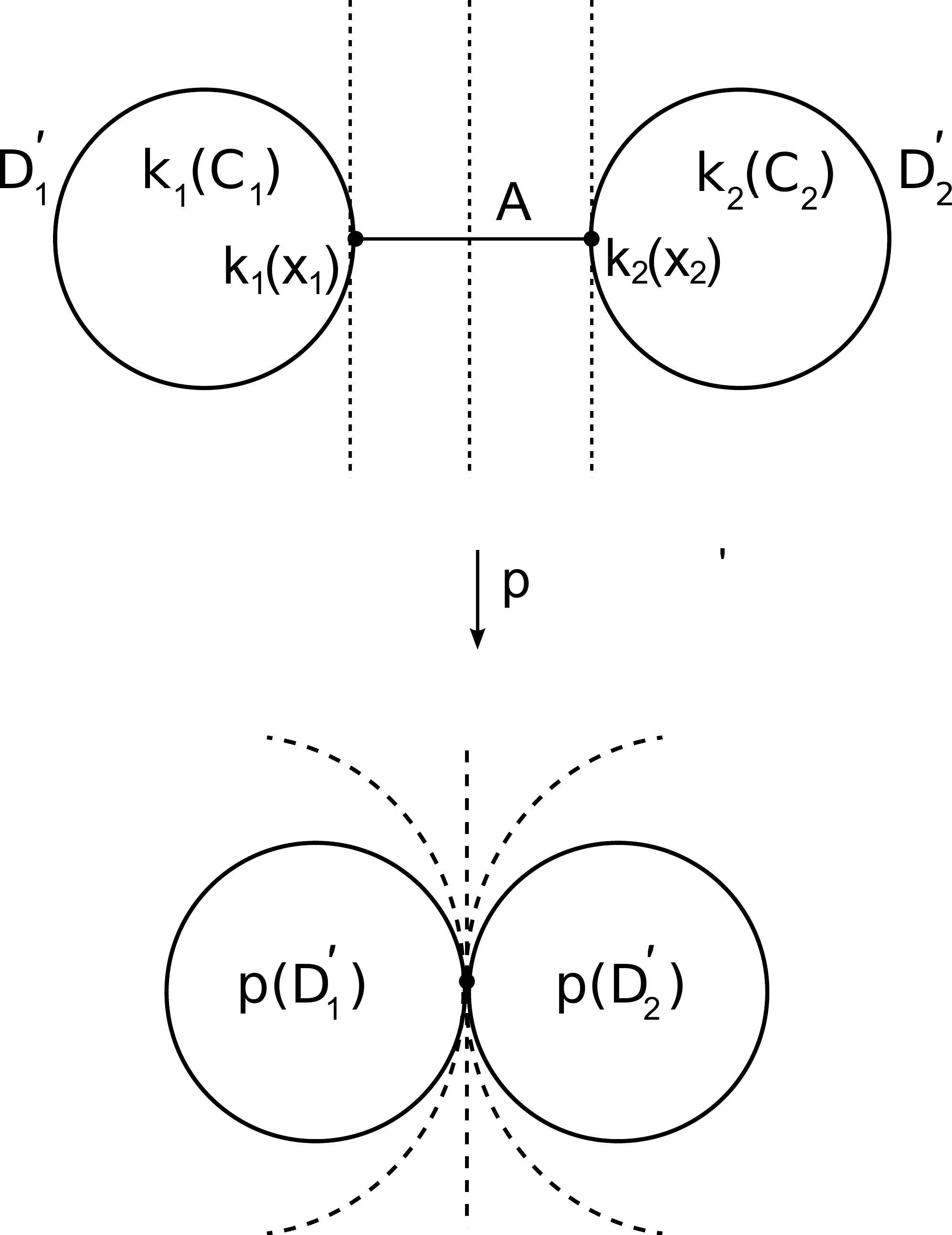}
\caption{Alternate Construction for Wedge }
\label{altwedge}
\end{figure}

Since $R^{3}/A\cong R^{3}$ and since $p\vert_{k_{i}(C_{i})}$ is 1-1,
it follows that $(C_{1},x_{1}) \bigvee^{\prime} (C_{2},x_{2})$ is homeomorphic to 
$(C_{1},x_{1}) \bigvee (C_{2},x_{2})$. It remains to check that the embeddings of
these homeomorphic spaces in $R^{3}$ are equivalent. This follows from the next lemma.

\begin{lemma}
\label{equivalentembeddinglemma}
$\left(
R^{3}, D_{1},D_{2},\textbf{0}\right) $
is homeomorphic to
$\left( R^{ 3}/A,
p(D^{\prime}_{1}),p(D^{\prime}_{2}),\right.$\\
$\left. p((1,0,0))\right)$.
\end{lemma}
\begin{proof}
There is a closed map $h:(R^{3},D_{1}^{\prime},D_{2}^{\prime},A)\rightarrow
(R^{3},D_{1},D_{2},\textbf{0})$ with the only nondegenerate point inverse
being $h^{-1}(\textbf{0})=A$. By standard topological results about quotient spaces, one can now establish the claim.
\end{proof}

\begin{theorem}
\label{pionetheorem}
Suppose $C_{1}$ and $C_{2}$ are as in Definition \ref{wedgedef2} and that
$\pi_{1}(C_{2}^{c})$ is trivial. Then $\pi_{1}\left(\strut^{\ } \left(\strut(C_{1},x_{1}) \bigvee (C_{2},x_{2})\right)^{c}\ \right)$
is isomorphic to $\pi_{1}(C_{1}^{c})$.
\end{theorem}
\begin{proof}
$\left(\strut(C_{1},x_{1}) \bigvee (C_{2},x_{2})\right)^{c}$ is homeomorphic to 
$\left(\strut(C_{1},x_{1}) \bigvee^{\prime} (C_{2},x_{2})\right)^{c}$ 
by Lemma \ref{equivalentembeddinglemma}. 
Also, $p:R^{3}\rightarrow R^{3}/A$ takes\\
$W=R^{3}\setminus \left(\strut k_{1}(C_{1})\cup A \cup k_{2}(C_{2})\right)$ 
homeomorphically onto\\
$\left(\strut(C_{1},x_{1}) \bigvee^{\prime} (C_{2},x_{2})\right)^{c}$ 
since $p$ is a quotient map and 
$p$ restricted to $W$ is 1-1.
 So it suffices to show that 
 $\pi_{1}\left(W\right)$
 is isomorphic to $\pi_{1}(C_{1}^{c})$. 
 
 Let $U=\{(x,y,z)\in W \vert x<1\}$ and 
 $V=\{(x,y,z)\in W \vert x> -1\}$. We will apply the Seifert -Van Kampen's Theorem to 
 $U, V,U\cap V$ and $W=U\cup V$. Note that $\pi_{1}(V)\simeq \pi_{1}(C_{2}^{c})$ which is trivial.
  Next, $\pi_{1}(U)\simeq \pi_{1}(C_{1}^{c})$. Also $\pi_{1}(U\cap V)\simeq \mathbb{Z}$.
 Since the inclusion-induced homomorphism from $U\cap V$ to $V$ is trivial, $\pi_{1}(W)\simeq
 \pi_{1}(C_{1}^{c})$.
 \end{proof}

\section{Local Genus of points in a Wedge}
\label{GenusWedgeSection}

The main result of this section, Theorem \ref{wedgetheorem},  states that local genus of points in Cantor sets is preserved by taking wedges, except at the wedge point. This allows us in the next section to distinguish between various wedges with complements having the same fundamental group. 

For the first two results below, we view the wedge of Cantor sets $C_{1}$ and $C_{2}$, 
$(C_{1},x_{1}) \bigvee (C_{2},x_{2})$, as in Definition \ref{wedgedef} and Figure \ref{Wedgefigure}  and for ease of notation, view $C_{1}$ and $C_{2}$ as subspaces of the wedge.

\begin{lemma}Let $W=(C_{1},x_{1}) \bigvee (C_{2},x_{2})$ be as in Definition \ref{wedgedef}. 
Then $g_{\{x_{1}=x_{2}\}}(W)=g_{\{x_{1}\}}(C_{1})+g_{\{x_{2}\}}(C_{2})$.
\end{lemma}
\begin{proof} This follows directly from Theorem 13 in \cite{Ze05} since there is a 3-cell containing $x_{1}=x_{2}$ in its interior satisfying the conditions needed for that theorem.
\end{proof}

\begin{lemma}Let $W=(C_{1},x_{1}) \bigvee (C_{2},x_{2})$ be as in Definition \ref{wedgedef} and let $p$ be a point in $C_{1}\setminus \{x_{1}\}$. Then $g_{\{x\}}(C_{1})\leq g_{\{x\}}(W)$.
\end{lemma}
\begin{proof} The result is obvious if $g_{\{x\}}(W)=\infty$. Assume $g_{\{x\}}(W)=k\in \mathbb{N}$.
By Lemma \ref{localgenuslemma}, there is a defining sequence $(M_{i})$ for $W$ such that for every $i$, $g(M_{i}^{\{x\}})=k$. Form a defining sequence $(N_{i})$ for $C_{1}$ as follows.
$N_{i}$ consists of the components of $M_{i}$ that intersect $C_{1}$. Then for every $i$, 
$g(N_{i}^{\{x\}})=k$ and thus again by Lemma \ref{localgenuslemma}, $g_{\{x\}}(C_{1})\leq k$.
\end{proof}

For the next result, we view the wedge of Cantor sets $C_{1}$ and $C_{2}$, 
$(C_{1},x_{1}) \bigvee (C_{2},x_{2})$, as in Definition \ref{wedgedef2} and Figure \ref{altwedge}. 
For ease of notation, we identify $C_{i}$ with $k_{i}(C_{i})$ so that

\centerline{$
 (C_{1},x_{1}) \bigvee^{\prime} (C_{2},x_{2})\equiv p\left(\strut(C_{1})\cup (C_{2})\right)
\equiv p\left(\strut(C_{1})\cup A\cup (C_{2})\right).
$}

\begin{lemma}Let $W=(C_{1},x_{1}) \bigvee^{\prime} (C_{2},x_{2})$ be as in Definition \ref{wedgedef2} and let $x$ be a point in $C_{1}\setminus \{x_{1}\}$ and let $x^{\prime}=p(x)$. Then $g_{\{x\}}(C_{1})\geq g_{\{x^{\prime}\}}(W)$.
\end{lemma}
\begin{proof}
Again, the result if obvious if $g_{\{x\}}(C_{1})=\infty$. 
Assume $g_{\{x\}}(C_{1})=k\in \mathbb{N}$. By Lemma \ref{localgenuslemma}, there is a defining sequence $(M_{i})$ for $C_{1}$ such that for every $i$, $g(M_{i}^{\{x\}})=k$. By starting the defining sequence at a late enough stage, we may assume that each component of each stage of the defining sequence is in the half space $\{(x,y,z)\vert x<- \frac{1}{2}\}$ in $R^{3}$, and we may assume the component of each $M_{i}$ that contains $x$ is distinct from the component of $M_{i}$ that contains $x_{1}$. Choose a defining sequence $N_{i}$ for $C_{2}$ so that each component of each stage of the defining sequence is in the half space $\{(x,y,z)\vert x> \frac{1}{2}\}$ in $R^{3}$.

We now adjust the defining sequence $(M_{i})$ replacing it by a defining sequence $(M^{\prime}_{i})$ so that the only component of $M^{\prime}_{i}$ that has nonempty intersection with $A$ is the component containing $x_{1}$, and so that for every $i$, $g({M^{\prime}}_{i}^{\{x\}})$ still is $k$.
Let $M_{(1,1)}$ be the component of $M_{1}=M_{n_{1}}$ containing $x_{1}$. Let $C_{(1,2)}$ be the Cantor set $C_{1}\setminus M_{(1,1)}$ and $C_{(1,1)} $ be the Cantor set $C_{1}\cap M_{(1,1)}$.
Let $d_{1}$ be the minimum of:\\
 \centerline{$\left\{d(C_{(1,2)},A),d(C_{(1,2)},M_{(1,1)}),
d(C_{(1,1)},Bd(M_{(1,1)}))\right\}$.} 
Choose a stage $n_{2}$ such that all components of $M_{n_{2}}$ have diameter less than $\dfrac{d_{1}}{2}$. Let $M^{\prime}_{1}$ consist of $M_{(1,1)}$ together with the components of $M_{n_{2}}$ that do not intersect $M_{(1,1)}$.

For the second step, repeat the above procedure on $M_{n_{2}}$, letting $M_{(2,1)}$ be the component of $M_{n_{2}}$ containing $x_{1}$, $C_{(2,2)}$ be the Cantor set $C_{1}\setminus M_{(2,1)}$ and $C_{(2,1)} $ be the Cantor set $C_{1}\cap M_{(2,1)}$.
Let $d_{2}$ be the minimum of:\\
 \centerline{$\left\{d(C_{(2,2)},A),d(C_{(2,2)},M_{(2,1)}),
d(C_{(2,1)},Bd(M_{(2,1)}))\right\}$.} 
Choose a stage $n_{3}$ such that all components of $M_{n_{3}}$ have diameter less than $\dfrac{d_{2}}{2}$. Let $M^{\prime}_{2}$ consist of $M_{(2,1)}$ together with the components of $M_{n_{2}}$ that do not intersect $M_{(2,1)}$.

Continuing inductively produces the desired defining sequence $(M^{\prime}_{i})$. Similarly, construct a defining sequence $(N^{\prime}_{i})$ for $C_{2}$ so that the only component of $N^{\prime}_{i}$ intersecting $A$ is the component containing $x_{2}$. Finally, choose a sequence of regular neighborhoods of $A$, $(P_{i})$, converging to $A$ so that for each $i$,
$W_{(i,1)}=M(i,1)\cup P_{i}\cup N_{(i,1)}$ form a manifold neighborhood of $A$ and converge to $A$.

The defining sequence for $W$ is then produced as follows. $W_{i}$ consists of $p(W_{(i,1)})$ together with $p(C)$ for all components $C$ of $M^{\prime}_{i}$ distinct from $M_{(i,1)}$ and all 
components $C$ of $N^{\prime}_{i}$ distinct from $N_{(i,1)}$. This defining sequence for $W$ has the property that for each $i$, the component of $W_{i}$ containing $x^{\prime}$ has genus $k$. It follows that $k\geq g_{\{x^{\prime}\}}(W)$ as required.

\end{proof}

The previous three lemmas together yield a proof of the following main theorem on genus of points in a wedge. Again, identify $C_{i}$ with $h_{i}(C_{i})$ for ease of notation.

\begin{theorem}
\label{wedgetheorem}
Let $W=(C_{1},x_{1}) \bigvee (C_{2},x_{2})$ be as in Definition \ref{wedgedef}.
Then 
\begin{itemize}
\item If $x\in C_{i}\setminus \{x_{i}\}$, then 
$g_{x}(C_{i})=g_{x}(W)$; and 
\item 
$g_{x_{1}}(W)
=g_{x_{1}}(C_{1})+g_{x_{2}}(C_{2})$.
\end{itemize}
\end{theorem}

\section{Main Result}
\label{ResultSection}
We are now ready to prove the main result, Theorem \ref{maintheorem}. 
\begin{proof}[Proof of Theorem \ref{maintheorem}]
Let $C$ be a Cantor set, and suppose there is some integer  $N\geq 3$ such that there are only finitely many points in $C$ of local genus $N$.

By Theorem \ref{rigidtheorem} and the remark following that theorem, for each  sequence 
$S=(n_1, n_2, \ldots)$
of integers in $S_{1}$, such that $n_i>2$, there exists a wild Cantor set in
$ R^3$, $X_{S}=C(S)$, and a countable dense set
$A=\{a_1, a_2, \ldots\} \subset X$ such that the following
conditions hold.
\begin{enumerate}
\item $g_x(X_{S}) \le 2$ for every $x\in X\setminus A$,
\item $g_{a_i}(X_{S})=n_i$ for every $a_i\in A$ and
\item $R^3\setminus X_{S}$ is simply connected.
\end{enumerate}
The construction in \cite{GaReZe06} yields the fact that the sets\\
 $A_{1}=\{a_{i}\vert i\text{ is odd}\}$ and
$A_{2}=$ $\{a_{i}\vert i\text{ is even}\}$ are also dense in $X(S)$. Choose an increasing sequence of integers $(m_{1},m_{2},\ldots)$ such that $m_{1}\geq 3$ and form a sequence of integers $S=(n_{1},n_{2},\ldots\}$ by specifying $n_{2i}=m_{i}$ and $n_{2i+1}=N$ for each $i$. The construction in 
\cite{GaReZe06} also yields the fact that $X(S)$ is rigidly embedded.

Let $Y_{S}=(C,x_{1}) \bigvee (X_{S},x_{2})$ for some points
$x_{1}\in C$ and $x_{2}\in X_{S}$ as in Definition \ref{wedgedef}.
 Condition (3) above together with Theorem \ref{pionetheorem} imply that
 $\pi_{1}(Y_{S}^{c})$
is isomorphic to $\pi_{1}(C^{c})$.

By Theorem \ref{wedgetheorem},  $Y_{S}$ has countably many points of genus $N$ and $C$ has only finitely many points of genus $N$. So $C$ and $Y_{S}$ are inequivalent Cantor sets.

Next suppose that $S$ and $S^{\prime}$ are formed from distinct increasing sequences of integers  as above. Suppose there were a homeomorphism $h$ of $R^{3}$ taking
$Y_{S}$ to $Y_{S^{\prime}}$. By local genus considerations,  a dense subset of the countable dense set of points in the copy of $X_{S}$ in $Y_{S}$ that have genus $N$ must be taken by $h$ into a dense subset of points of $X_{S^{\prime}}$ in $Y_{S^{\prime}}$ that have genus $N$. 
Also a dense subset of the countable dense set of points in the copy of $X_{S^{\prime}}$ in $Y_{S^{\prime}}$ that have genus $N$ must be taken by $h^{-1}$ into a dense subset of the points of $X_{S}$ in $Y_{S}$ that have genus $N$.  It follows
that $h$ takes the copy of $X_{S}$ in $Y_{S}$ onto the copy of $X_{S^{\prime}}$ in $Y_{S^{\prime}}$. But this contradicts the fact that there is either a genus that occurs among points of $X(S)$ that does not occur among points of $X(S^{\prime})$, or vice versa.

Thus every increasing sequence of integers $\geq 3$ yields a different Cantor set $Y_{S}$. Since there are uncountably many such sequences, there are uncountably many such examples for each Cantor set $C$ as in the theorem. The result now follows.
\end{proof}

\begin{corollary} The Cantor sets 
constructed in the proof of
Theorem \ref{maintheorem}, $Y_{S}=(C,x_{1}) \bigvee^{\prime} (X_{S},x_{2})$, are \emph{not} splittable as $C\cup A$.
\end{corollary}
\begin{proof}
If these Cantor sets were splittable, $C$ would be in a $3$-cell $D_{1}$ disjoint from a $3$-cell $D_{2}$ containing $Y_{S}
\setminus C$ = $X_{S}\setminus x_{2}$. This would be a contradiction.

\end{proof}

 \begin{corollary}
\label{rigidcorollary}
There are uncountably many rigid Cantor sets with complement nonsimply connected
that have the same fundamental group of the complement.
\end{corollary}
\begin{proof}
It suffices to take for $C$ in the proof of  Theorem \ref{maintheorem} 
any of the rigid Antoine Cantor sets
of local genus 1 everywhere. See \cite{Wright3} for a description of these Cantor sets.
One can also take for $C$ any of the rigid Cantor sets constructed in
\cite{GaReZe06} since they take on certain genera only once. \end{proof}

Theorem \ref{maintheorem} and Corollary \ref{rigidcorollary} answer both questions from the beginning of the paper.

\section{Application to $3$-Manifolds}

If we work in $S^3$ instead of $R^3$, the following lemma is a consequence of results about Freudenthal compactifications and theory of ends. (see and \cite{Freudenthal}, \cite{Dickman}, and
\cite{Siebenmann}) For completeness, we provide a proof based on defining sequences. 

\begin{lemma}\label{endlemma}
Let $C$ and $D$ be Cantor sets (or more generally, any compact $0$-dimensional sets) in $\mathbb{R}^{3}$. Suppose there is a homeomorphism $h:
\mathbb{R}^{3} \setminus C\rightarrow \mathbb{R}^{3}\setminus D$. Then $h$ extends to a homeomorphism $\bar{h}:(\mathbb{R}^{3},C)\rightarrow (\mathbb{R}^{3},D)$. In particular, $C$ is homeomorphic to $D$ and $C$ and $D$ are equivalently embedded.
\end{lemma}
\begin{proof}
Let $(M_{i})$ be a defining sequence for $C$. Suppose 
\\$M_i=\{M_{(i,1)}, M_{(i,2)},\ldots M_{(i,n(i))}\}$. 
Let $N_{(i,j)}$ be the bounded component of $\mathbb{R}^{3}\setminus h(Bd(M_{(i,j)}))$, and let $N_i=
\{N_{(i,1)}, N_{(i,2)},\ldots N_{(i,n(i))}\}$. The claim is that $(N_{i})$ is a defining sequence for $D$
so that the nested sequence in $M_{i}$ associated with a point $c\in C$ corresponds to a nested sequence in $N_{i}$ corresponding to a point $d\in D$. This forces $\bar{h} ( c ) $ to be defined to be $d$. Now $\bar{h}$ defined in this way is continuous and 1-1 because of the definition of defining sequences.  The fact that a similar construction can be done using $h^{-1}$ shows that $\bar{h}$ takes $C$ onto $D$. 

It is clear that any nested sequence in $(N_{i})$ has intersection compact, connected and in $D$, so consists of a single point of $D$. This establishes the claim at the beginning of the preceding paragraph and completes the proof.
\end{proof}

We now provide the proof of Theorem \ref{endtheorem}.
\begin{proof}
The $3$-manifolds are the complements in $R^{3}$ of the Cantor sets constructed in the proof of Theorem \ref{maintheorem} in Section \ref{ResultSection}. These are all nonhomeomorphic by
Lemma \ref{endlemma}.
\end{proof}

See \cite{Kister-McMillan} and  \cite{McMillan} for earlier results on nonhomeomorphic $3$-manifolds with the same fundamental group.

\section{Related Questions}
\label{QuestionSection}

The techniques in this paper lead to a number of open questions.

\begin{question}
Are there examples of inequivalent Cantor sets with the same fundamental group
of the complement for which the genus of all Cantor sets involved is bounded? The constructions described in this paper yield wedges that have points or arbitrarily large genus.
\end{question}

\begin{question}
Given a Cantor set that does not meet the criteria of Theorem \ref{maintheorem},
are there inequivalent Cantor sets with the same fundamental group of the complement?
Note that a Cantor set not covered by these results would have an infinite number of points
of every genus greater than or equal to three.
\end{question}
\begin{question}
Do the results of Theorem \ref{endtheorem} remain true if the restriction on local genus of points in the Cantor set is removed?
\end{question}
 \begin{question}
Does Theorem \ref{maintheorem}
remain true if the Cantor set $C$ has only a finite number of points of genus $1$, or of genus $2$, or only a finite number of points of genus $\infty$?
\end{question}
\begin{question}
Can the construction in \cite{GaReZe06} be modified to produce specific points in a countable dense subset that have infinite genus rather than certain specified finite genera?
\end{question}

A Cantor set $C$ is said to be \emph{strongly homogeneously embedded} in $R^{3}$ if every self homeomorphism 
of $C$ extends to a self homeomorphism of $R^{3}$. Define the \emph{embedding homogeneity group}
of the Cantor set to be the group of self homeomorphisms that extend to homeomorphisms of $R^{3}$. Rigid Cantor sets have trivial embedding homogeneity group.

\begin{question}
Given a finitely generated abelian group $G$, is there a Cantor set $C$ in $R^{3}$
with embedding homogeneity group $G$?
\end{question}

\begin{question}
Given a finite abelian group $G$, is there a Cantor set $C$ in $R^{3}$
with embedding homogeneity group $G$?
\end{question}

\begin{question}
What kinds of groups arise as embedding homogeneity groups of Cantor sets?
\end{question}

\section{Acknowledgements}
 The authors were supported in part by the Slovenian Research Agency grants
 P1-0292-0101, J1-2057-0101 and BI-US/11-12-023.
 The first author was supported in part by the National Science Foundation grants 
 DMS 0852030 and DMS 1005906.

\end{document}